\documentclass[11pt]{amsart}

\usepackage{style}

\usepackage[foot]{amsaddr}

\title{Uniform approximation of continuous couplings}
\author{Ugo Bindini \and Tapio Rajala}
\address{Department of Mathematics and Statistics, University of Jyväskylä}
\date{\today}

\begin{document}
	
\maketitle

\begin{abstract}
	We study the approximation of non-negative multi-variate couplings in the uniform norm while matching given single-variable marginal constraints.
\end{abstract}

\section{Introduction}

For some integer $N \geq 2$, let $(X_1, d_1, \mu_1), \dotsc, (X_N, d_N, \mu_N)$ be probabilistic metric spaces and $(X, d, \mu)$ the product space $X = X_1 \times \dotsb \times X_N$ equipped with the product measure $\mu = \mu_1 \otimes \dotsb \otimes \mu_N$ and the distance
\[ d(x,x') = \sum_{j=1}^N d_j(x_j,x_j'). \]

We will call a function $\rho_j \in C(X_j)$ a \emph{marginal} if it is non-negative and $\int_{X_j} \rho_j d\mu_j = 1$. Given  $\rho_1, \dotsc, \rho_N$ marginals, we aim to study the space of continuous \emph{couplings} $\Pi(\rho_1, \dotsc, \rho_N)$ defined as
\[ \Pi(\rho_1, \dotsc, \rho_N) \eqdef \gra{P \in C(X) \st P \geq 0 \st \pi_jP= \rho_j \forall j=1, \dotsc, N}, \]
where $\pi_jP$ denotes the $j$-th marginal of $P$, obtained by integrating $P$ with respect to all the reference measures $\mu_1, \dotsc, \mu_N$ except $\mu_j$:
\[ \pi_jP(x_j) = \int_{X_1 \times \dotsb \times \widehat{X_j} \dotsb \times X_N} P(x_1, \dotsc, x_N) d\mu_1(x_1) \dotsm \widehat{d\mu_j(x_j)} \dotsm d\mu_N(x_N). \]
Notice that $\Pi(\rho_1, \dotsc, \rho_N)$ is not empty, because the product $\rho_1 \otimes \dotsb \otimes \rho_N$ defined by
\[ (\rho_1 \otimes \dotsb \otimes \rho_N)(x) \eqdef \rho_1(x_1) \dotsm \rho_N(x_N) \]
is a continuous coupling.

Continuous couplings arise as solutions of the regularized Kantorovich multi-marginal optimal transport problem
\[ \mathcal{K}_h(\rho_1, \dotsc, \rho_N) = \inf_{P \in \Pi(\rho_1, \dotsc, \rho_N)} \gra{\int_X c(x) P(x) d\mu(x) + h F(P)}, \]
where $c \colon X \to \overline{\R}$, $h > 0$ and $F$ is a regularizing functional, typically strictly convex. Examples include the entropic regularization (\cf \cite{carlier2017convergence, clason2021entropic, gerolin2020multi}), the quadratic regularization (\cf \cite{essid2018quadratically, lorenz2021quadratically}) and others (\cf \cite{flamary2014optimal, korman2015optimal}). Here, the continuous marginals $\rho_j$ play the role of a Radon-Nikodym densities with respect to the reference measures $\mu_j$.

This class of problems was introduced for computational reasons as perturbations of the Kantorovich multi-marginal optimal transport problem
\[ \mathcal{K}(\rho_1, \dotsc, \rho_N) = \inf_{P \in \Pi(\rho_1, \dotsc, \rho_N)} \int_X c(x) P(x) d\mu(x), \]
as strict convexity of $\mathcal{K}_h$ provides good convergence properties for the sequence of minimizers (\cf \cite{cuturi2013sinkhorn, santambrogio2015optimal, villani2008optimal}). Minimization problems of the form $\mathcal{K}_h$ emerges also from the setting of Density Functional Theory (\cf \cite{bindini2017optimal, buttazzo2012optimal, cotar2013density, cotar2018smoothing, lewin2018semi, lewin2019optimal}), where the kinetic energy plays the role of the regularizing functional. 

It is important to analyze the continuity properties of $\mathcal{K}_h$ in order to apply variational tools (\eg, $\Gamma$-convergence) to show that in the limit $h \to 0$ we get convergence to $\mathcal{K}$ in a suitable topology. This raises the following question: given a coupling $P$ and marginals $\rho_1, \dotsc, \rho_N$ such that $\norm{\pi_j P - \rho_j}$ is small, does there exist $P' \in \Pi(\rho_1, \dotsc, \rho_N)$ such that $\norm{P - P'}$ is small? In other words, can we approximate a coupling $P$ while matching a given marginal constraint?

This problem has been solved recently for the $L^2$ norm and the $W^{1,2}$ norm on the Euclidean space (see \cite{bindini2019from}), with a construction which can be generalized to every exponent $1 < p < \infty$ in view of \cite{bindini2020smoothing}. The point of this short note is to extend this result to the uniform norm $L^\infty$ for continuous couplings on generic metric measure spaces.

The difficulty in the approximation lies in the non-negativity assumptions on the couplings: if $P'$ were allowed to be any continuous function, we could set
\[ P' = P + \rho_1 \otimes \dotsb \otimes \rho_N - \pi_1P \otimes \dotsb \otimes \pi_N P, \]
which, except for being non-negative, has the required properties (continuity and closeness to $P$).

Our main result is the following.

\begin{thm} \label{thm:main}
	Suppose that $\spt \mu$ is compact. Let $P$ be a coupling, and $\rho_j \in C(X_j)$ be marginals for $j = 1, \dotsc, N$. Then, for every $\eps > 0$ there exist $\sigma(\eps) > 0$ such that, if $\lVert \pi_j P - \rho_j \rVert_{L^\infty(\mu_j)} < \sigma(\eps)$ for every $j = 1, \dotsc, N$ then there exists $P' \in \Pi(\rho_1, \dotsc, \rho_N)$ with $\norm{P - P'}_{L^\infty(\mu)} < \eps$.
\end{thm}

We will prove \autoref{thm:main} in Section \ref{section:main}, where we also provide a precise expression for the threshold $\sigma(\eps)$. Then, in Section \ref{section:remarks} we will show that \autoref{thm:main} is sharp when the base metric spaces are Euclidean with the Lebesgue measure, by proving the following remarks.
\begin{enumerate}[A.]
	\item The hypothesis of compactness for the support of $\mu$ cannot be removed.
	\item The continuity hypothesis on $P$ cannot be removed, even assuming all the marginals to be continuous.
	\item If $P$ is Lipschitz, the expression of $\sigma(\eps)$ which emerges from Section \ref{section:main} is sharp.
\end{enumerate}

Unfortunately, when $P$ is not Lipschitz, the threshold $\sigma(\eps)$ cannot be computed exactly as a workable expression. Due to this, the sharpness of \autoref{thm:main} is provided only in the Lipschitz case.
%

\subsection*{Acknowledgements} We are grateful to prof. L. De Pascale (Univ. Firenze) for the fruitful discussions and to the financial support of the Academy of Finland (grant no. 314789).

\section{Proof of \autoref{thm:main}} \label{section:main}

In order to lighten the notation, for $j = 1, \dotsc, N$ we introduce the spaces $\hat{X}_j \eqdef X_1 \times \dotsb \times X_{j-1} \times X_{j+1} \times \dotsb \times X_N$ endowed with the probability reference measure $\hat{\mu}_j \eqdef \mu_1 \otimes \dotsb \otimes \mu_{j-1} \otimes \mu_{j+1} \otimes \dotsb \otimes \mu_N$.

If $P$ is any continuous function on $X$, a modulus of continuity for $P$ is given by the expression
\[ \omega(r) \eqdef \sup_{x,y \in X} \gra{\abs{P(x) - P(y)} \st d(x,y) \leq r}. \]
Observe that $\omega \colon [0,+\infty) \to [0,+\infty)$ is a non-decreasing function with $\omega(0) = 0$. We will need the inverse of the modulus of continuity given by 
\[ \omega^{-1}(t) = \sup_{x,y \in X} \gra{d(x,y) \st \abs{P(x) - P(y)} \leq t}. \]

For $j = 1, \dotsc, N$ we also introduce the radial maximal function $f_j \colon [0, +\infty) \to [0, 1]$
\begin{equation} \label{eq:radial-function}
f_j(r) \eqdef \inf_{x_j \in \spt \mu_j} \mu_j(B(x_j, r)).
\end{equation}
For every $j$, $f_j$ is a non-decreasing function with $f(0) = 0$, $f_j(r) > 0$ for $r > 0$ if $\spt \mu$ is compact.

The following result allows to control in a quantitative fashion the value of $P(x_1, \dotsc, x_N)$ depending on the values $\gra{\pi_jP(x_j)}_{j = 1, \dotsc, N}$.

\begin{lemma} \label{lemma:eps-sigma}
	Let $P$ be a coupling, $\omega$, $f_j$ as above. Then for every $\eps > 0$, for every $j = 1, \dotsc, N$, for $\mu_j$-a.e. $x_j \in X_j$,
	\[ \pi_jP(x_j) \leq \sigma_j(\eps) \implies P(y_1, \dotsc, x_j, \dotsc, y_N) \leq \eps \quad \text{for $\hat \mu_j$-a.e. $y \in \hat{X}_j$,} \]
	where
	\begin{equation} \label{eq:sigma}
		\sigma_j(\eps) = \eps \sup_{\theta \in (0,1)} (1-\theta) \prod_{k \neq j} f_k\pa{\frac{\omega^{-1}(\theta\eps)}{N-1}}.
	\end{equation}
\end{lemma}

\begin{proof}
	Let $x = (x_1, \dotsc, x_N) \in \spt \mu$ such that $P(x) > \eps$ (if there is no such point, the thesis is trivially true) and $\theta \in (0,1)$. If $r(\eps) \eqdef \omega^{-1}(\theta\eps)$, we have
	\[ P(x_1', \dotsc, x_j, \dotsc, x_N') > (1 - \theta) \eps \quad \text{if $x_k' \in B\pa{x_k, \frac{r(\eps)}{N-1}}$ for $k \neq j$}, \]
	since for such points we have
	\[ d((x_1', \dotsc, x_j, \dotsc, x_N'),x) = \sum_{k \neq j} d_k(x_k', x_k) < r(\eps). \]
	Now, using the definition \eqref{eq:radial-function},
	\begin{ieee*}{rCl}
		\pi_j P(x_j) &=& \int_{\hat{X}_j} P(x_1', \dotsc, x_j, \dotsc, x_N') d\hat{\mu}_j(x') > \int_{\prod\limits_{k \neq j} B\pa{x_k, \frac{r(\eps)}{N-1}}} (1 - \theta)\eps d\hat\mu_j(x') \\
		&\geq& \eps (1 - \theta) \prod_{k \neq j} f_k\pa{\frac{\omega^{-1}(\theta\eps)}{N-1}}.
	\end{ieee*}
	
	By passing to the supremum in $\theta$ on the right hand side we get the thesis.
\end{proof}

\begin{remark} \label{remark:sigma}
	Observe that $\sigma_j(\eps) = \eps \phi_j(\eps)$, where $\phi_j$ is a non-decreasing function (supremum of non-decreasing functions). If $\spt \mu$ is compact, $\phi_j(\eps) > 0$ for $\eps > 0$, which implies that $\sigma$ is a positive strictly increasing function on $(0,+\infty)$. In particular, the following property holds:
	\begin{equation} \label{eq:sigma-convex}
		\frac{s}{\sigma_j(s)} \leq \frac{\eps}{\sigma_j(\eps)} \quad \forall 0 < \eps \leq s.
	\end{equation}
\end{remark}

\begin{remark}
	If $\spt \mu$ is compact, due to the invertibility of $\sigma_j$, an equivalent way to write the thesis of \autoref{lemma:eps-sigma} is the following: for every $s \geq 0$,
	\begin{equation} \label{eq:inv-sigma}
		\pi_jP(x_j) \leq s \implies P(y_1, \dotsc, x_j, \dotsc, y_N) \leq \sigma_j^{-1}(s) \quad \forall y \in \hat X_j.
	\end{equation}
\end{remark}

A preliminary step towards the proof of \autoref{thm:main} is to modify only one marginal to a given target.

\begin{thm} \label{thm:one-marginal}
	Suppose that $\spt \mu$ is compact. Let $P$ be a coupling and $\rho_j \in C(X_j)$ a marginal for some $j \in \gra{1, \dotsc, N}$. Then there exist constants $\kappa_j, K_j > 0$ depending on $\pi_j P$ such that, if $0 < \eps < \kappa_j$ and $\norm{\rho_j - \pi_j P}_{L^\infty(\mu_j)} < \sigma_j(\eps)$, then there exists a coupling $P'$ such that
	\begin{itemize}
		\item $\pi_jP' = \rho_j$, $\pi_k P' = \pi_kP$ for every $k \neq j$;
		\item $\norm{P - P'}_{L^\infty(\mu)} < K_j \eps$.
	\end{itemize}
\end{thm}

\begin{proof}
	For simplicity of notation we will take $j = 1$ and denote $\sigma(\eps) = \sigma_1(\eps)$, but the argument works in the same way for every $j$. Fix $\eps > 0$ and consider the coupling
	\[
	P_\eps(x) = P(x) \frac{\max(0, \pi_1P(x_1)-\sigma(\eps))}{\pi_1P(x_1)}
	\]
	where $0/0 = 0$. Then $\pi_1P_\eps(x_1) = \max(0, \pi_1P(x_1) - \sigma(\eps))$ and $\pi_jP_\eps(x_j) \leq \pi_jP$.
	
	We have
	\[
	P(x) - P_\eps(x) = P(x) \frac{\min(\pi_1P(x_1), \sigma(\eps))}{\pi_1P(x_1)}.
	\]
	By \autoref{lemma:eps-sigma}, if $\pi_1P(x_1) < \sigma(\eps)$ and $x_1 \in \spt \mu_1$,
	\[ \abs{P(x_1, y) - P_\eps(x_1, y)} < \eps \quad \text{for $\hat{\mu}_2$-a.e. $y \in \hat{X}_2$;} \]
	on the other hand, if $\pi_1P(x_1) \geq \sigma(\eps)$ and $x_1 \in \spt \mu_1$, letting $s = \sigma^{-1}(\pi_1P(x_1))$ in \eqref{eq:sigma-convex} and using \eqref{eq:inv-sigma},
	\[ \abs{P(x_1, y) - P_\eps(x_1, y)} \leq \sigma_1^{-1}(\pi_1P(x_1)) \frac{\sigma(\eps)}{\pi_1P(x_1)} \leq \eps \quad \forall y \in \hat X_2. \]
	
	All in all, 
	\[ \norm{P - P_\eps}_{L^\infty(\mu)} \leq \eps. \]
	
	Let $2c_1 = \mu_1(\gra{\pi_1P > 0}) > 0$, and let $\kappa_1$ such that $\mu_1(\gra{\pi_1P > \sigma(\eps)}) \geq c_1$ whenever $\eps < \kappa_1$. For such $\eps$ we have
	\begin{ieee*}{rCl}
		m(\eps) &\eqdef& \int_X P(x) - P_\eps(x) d\mu(x) = \int_{X_1} \min(\pi_1P(x_1), \sigma(\eps)) d\mu_1(x_1) \\
		&\geq& \int_{\gra{\pi_1P > \sigma(\eps)}} \sigma(\eps) d\mu_1(x_1) = \sigma(\eps) \mu_1(\gra{\pi_1P > \sigma(\eps)}) \geq c_1\sigma(\eps).
	\end{ieee*}
	
	Now we define the coupling
	\begin{ieee*}{rCl}
	P'(x_1, \dotsc, x_N) &=& P_\eps(x_1, \dotsc, x_N) + \frac{\rho_1(x_1) -  \pi_1P_\eps(x_1)}{m(\eps)} \times \\
	&& \int_{X_1} \pa{P(x_1, \dotsc, x_N) - P_\eps(x_1, \dotsc, x_N)} d\mu_1(x_1).
	\end{ieee*}
	
	Since for every $j = 2, \dotsc, N$
	\begin{ieee*}{rCl}
		m(\eps) &=& \int_{X_j} \pi_jP - \pi_jP_\eps d\mu_j = \int_X P - P_\eps d\mu = \int_{X_1} \rho_1 - \pi_1P_\eps d\mu_1,
	\end{ieee*}
	we have $P' \in \Pi(\rho_1, \pi_2P, \dotsc, \pi_NP)$. Moreover, if $\norm{\rho_1 - \pi_1P}_{L^\infty(\mu_1)} < \sigma(\eps)$, we have
	\[ \norm{\rho_1 - \pi_1P_\eps}_{L^\infty(\mu_1)} \leq \norm{\rho_1 - \pi_1P}_{L^\infty(\mu_1)} + \norm{\pi_1P - \pi_1P_\eps}_{L^\infty(\mu_1)} < 2\sigma(\eps), \]
	while clearly
	\[ \norm{\int_{X_1} P(x_1, \dotsc, x_N) - P_\eps(x_1, \dotsc, x_N) d\mu_1(x_1)}_{L^\infty(\hat{\mu}_1)} \leq \norm{P - P_\eps}_{L^\infty(\mu)} \leq \eps. \]
	
	Finally, $P'$ is a non-negative function since, if $\pi_1P(x_1) < \sigma(\eps)$,
	\[ \rho_1(x_1) - \pi_1P_\eps(x_1) = \rho_1(x_1) \geq 0, \]
	while, if $\pi_1P(x_1) \geq \sigma(\eps)$,
	\[ \rho_1(x_1) - \pi_1P_\eps(x_1) = \rho_1 - \pi_1P(x_1) + \sigma(\eps) > 0. \]
	
	Therefore, $\norm{P' - P}_{L^\infty(\mu)} < \norm{P - P_\eps}_{L^\infty(\mu)} + \frac{2}{c_1}\eps \leq \pa{1 + \frac{2}{c_1}}\eps$.
\end{proof}

As a consequence we get the following equivalent formulation of \autoref{thm:main}.

\begin{thm} \label{thm:multi-marginal}
	Suppose that $\spt \mu$ is compact. Let $P$ be a coupling and $\rho_1, \dotsc, \rho_N$ marginals. Then there exist constants $\kappa,K > 0$ depending on $P$ such that, for every $0 < \eps < \kappa$, if $\norm{\rho_j - \pi_jP}_{L^\infty(\mu_j)} < \sigma_j(\eps)$ for all $j$, then there exists $P' \in C(X)$ such that
	\begin{itemize}
		\item $\pi_jP' = \rho_j$ for every $j = 1, \dotsc, N$;
		\item $\norm{P - P'}_{L^\infty(\mu)} < K\eps$.
	\end{itemize}
\end{thm}

\begin{proof}
	Let $\kappa = \min \gra{\kappa_1, \dotsc, \kappa_N}$, where $\kappa_j$ is given by \autoref{thm:one-marginal}. If $\eps < \kappa$ we have a coupling $P_1$ such that $\pi_1P_1 = \rho_1$, $\pi_jP_1 = \pi_jP$ for $j \geq 2$ and $\norm{P_1 - P}_{L^\infty(\mu)} < K_1\eps$. To keep in mind that $K_1$ depends only on $\pi_1P$, let us write it as $K_1 = K(\pi_1P)$.
	
	By applying \autoref{thm:one-marginal} to $P_1$, we can now find a coupling $P_2$ such that $\pi_jP_2 = \rho_j$ for $j = 1, 2$, $\pi_jP_2 = \pi_jP$ for $j \geq 3$ and
	\[ \norm{P_2 - P_1}_{L^\infty(\mu)} < K(\pi_2 P_1) \eps = K(\pi_2 P) \eps, \]
	where we exploited the fact that $\pi_2P_1 = \pi_2 P$.
	
	Therefore,
	\[ \norm{P_2 - P}_{L^\infty(\mu)} < (K(\pi_1P) + K(\pi_2P))\eps = (K_1 + K_2) \eps. \]
	
	By continuing in the same way we eventually get $P' \in C(X)$ such that $\pi_jP' = \rho_j$ for every $j = 1, \dotsc, N$ and
	\[ \norm{P' - P} < K\eps, \quad K \eqdef \sum_{j = 1}^N K_j. \qedhere \]
\end{proof}

This allows to approximate a continuous coupling in the uniform norm while matching a marginal contstraint, as follows.

\begin{corollary}
	Suppose that $\spt \mu$ is compact. Let $P$ be a coupling and $(\rho^n_j)_{n \in \N}$ for $j = 1, \dotsc, N$ sequences of marginals such that $\rho_j^n \to \pi_jP$ in $L^\infty(\mu_j)$ for every $j$. Then there exists a sequence $(P_n)_{n \in \N}$ of couplings such that $\pi_j P_n = \rho_n$ for every $j,n$ and $P_n \to P$ in $L^\infty(\mu)$.
\end{corollary}

\begin{proof}
	Recall the invertibility of $\sigma_j$ for every $j$ (\autoref{remark:sigma}), and define
	\[ \eps_n = \max_{j = 1, \dotsc, N} \sigma_j^{-1}\pa{\norm{\rho_j^n - \pi_jP}_{L^\infty(\mu_j)}}. \]
	
	Since the $\sigma_j$'s are increasing, we have $\sigma_j(\eps_n) \geq \lVert \rho_j^n - \pi_jP \rVert_{L^\infty(\mu_j)}$, hence by \autoref{thm:multi-marginal} we can find $P_n$ such that $\norm{P_n - P}_{L^\infty(\mu)} \leq K\eps_n$. Finally,
	\[ \lim_{t \to 0^+} \sigma_j^{-1}(t) = 0 \implies \lim_{n \to \infty} \eps_n = 0, \]
	as wanted. 
\end{proof}

\section{Remarks on \autoref{thm:main}} \label{section:remarks}

As stated in the Introduction, we give in this section some remarks about the sharpness of our main result. For $b,h > 0$ we define the tent function $\Delta \colon \R \to \R$ of base $2b$ and height $h$ as
\begin{equation} \label{eq:tent-function}
\Delta(b,h; t) = \begin{cases} h\pa{1 + \frac{t}{b}} & t \in [-b, 0] \\ h\pa{1 - \frac{t}{b}} & t \in [0, b] \\
0 & \text{otherwise,} \end{cases}
\end{equation}

\subsection*{A. The compactness of $\spt \mu$ is necessary}

This example shows that the compactness of $\spt \mu$ is necessary. With $N = 2$, let $X_1 = \R$, $\mu_1 \in \Prob(X_1)$ not compactly supported and $X_2 = [-\frac{1}{2}, \frac{1}{2}]$ with the Lebesgue measure.

Fix a monotone sequence of points $(y_n)_{n \in \N}$ such that $y_n \in \spt \mu_1$ and $y_n \nearrow +\infty$ or $y_n \searrow -\infty$, with $\abs{y_{n+1} - y_n} \geq 1$. Let $\Delta(b, h; x)$ be the tent function defined in \eqref{eq:tent-function}, and $P \in C(\R \times [-1,1])$ the coupling
\[ P(x_1, x_2) = c \sum_{n = 2}^\infty \Delta\pa{\frac{1}{n}, 1; x_1 - y_n} \Delta\pa{\frac{1}{n}, 1; x_2}, \]
where $c > 0$ is such that $\int_X P d\mu = 1$. Observe that
\[ \pi_1 P = c \sum_{n = 2}^\infty \Delta\pa{\frac{1}{n}, \frac{1}{n}; x_1 - y_n}. \]

Given $n \in \N$, it is possible to take a marginal $\rho_n$ such that $\norm{\rho - \pi_1P}_\infty \leq \frac{c}{n}$ by letting $\rho_n \equiv 0$ on the interval $\left[y_n - \frac{1}{n}, y_n + \frac{1}{n} \right]$ (and adjusting it outside so that $\int \rho_n = 1$). In particular, if $P'$ is any coupling with $\pi_1P' = \rho_n$, we have $P' \equiv 0$ on the strip $\left[y_n - \frac{1}{n}, y_n + \frac{1}{n} \right] \times \left[-\frac{1}{2}, \frac{1}{2}\right]$. Using the fact that $y_n \in \spt \mu_1$, this in turn implies
\[ \norm{P - P'}_{L^\infty(\mu)} \geq \norm{P \cdot \One_{[y_n - \frac{1}{n}, y_n + \frac{1}{n}] \times [-\frac{1}{2}, \frac{1}{2}]}}_{L^\infty(\mu)} = c, \]
which shows that the coupling $P$ cannot be approximated as in \autoref{thm:main}.

The same example can be extended to every $N$ just by considering $N$-variate products of tent functions on the space $\R \times [-\frac{1}{2}, \frac{1}{2}]^{N-1}$

\subsection*{B. The continuity of $P$ is necessary}

A very similar idea as in the previous example can be exploited to show that the continuity assumption on $P$ is also necessary. On $[0,1] \times [0,1]$, equipped with the Lebesgue measure, one can consider the symmetric coupling
\[ P(x_1, x_2) = \sum_{n = 2}^\infty \Delta\pa{\frac{1}{2^n}, 1; x_1 - \frac{3}{2^n}} \Delta\pa{\frac{1}{2^n}, 1; x_2 - \frac{3}{2^n}}, \]
which is discontinuous in $(0,0)$, but has continuous marginals
\[ \pi_1 P(x_1) = \sum_{n = 2}^\infty \frac{1}{2^n} \Delta\pa{\frac{1}{2^n}, 1; x_1 - \frac{3}{2^n}}. \]

Again, given $n$, it is possible to take a continuous marginal $\rho_n$ such that $\norm{\rho_n - \pi_1 P}_\infty \leq \frac{1}{2^n}$ by letting $\rho_n \equiv 0$ on the interval $\left[ \frac{1}{2^{n-1}}, \frac{1}{2^n} \right]$ (and adjusting it outside so that $\int \rho_n = 1$). Every coupling $P'$ with $\pi_1 P' = \rho_n$ must be equal to zero in the strip $\left[ \frac{1}{2^{n-1}}, \frac{1}{2^n} \right] \times [0,1]$, therefore
\[ \norm{P - P'}_\infty \geq \sup \gra{P(x) \st x \in \left[ \frac{1}{2^{n-1}}, \frac{1}{2^n} \right] \times [0,1]} = 1, \]
which shows that the coupling $P$ cannot be approximated as in \autoref{thm:main}.

As before, the same example extends to every $N$ by considering $N$-variate products of tent functions on $[0,1]^N$.

\subsection*{C. Sharpness of the threshold $\sigma(\eps)$ in \autoref{thm:one-marginal}}

For Lipschitz couplings on an interval with the Lebesgue measure, \autoref{thm:one-marginal} is sharp in the following sense: there exists a constant $C(N)$ such that, for every $0 < \eps < \frac{1}{4}L$, there exist $P \in C([0,1]^N)$, $\rho_1 \in C([0,1])$ such that
\begin{enumerate}[(i)]
	\item $P$ is $L$-Lipschitz;
	\item $\sigma_1(\eps) \leq \norm{\rho_1 - \pi_1P} \leq C(N) \sigma_1(\eps)$;
	\item for every $P'$ with $\pi_1 P' = \rho_1$, $\norm{P - P'}_\infty \geq  \eps$.
\end{enumerate}

First of all observe that in this case, exploiting $\omega(r) = Lr$, we have
\[ \sigma(\eps) = \sup_{\theta \in [0,1]} (1-\theta)\eps\pa{\frac{\theta\eps}{L(N-1)}}^{N-1} = c(N) \frac{\eps^N}{L^{N-1}}, \]
where explicitly
\[ c(N) = \frac{1}{(N-1)^{N-1}} \sup_{\theta \in (0,1)} (1-\theta)\theta^{N-1} = \frac{1}{N^N}. \]

Given $0 < \eps < \frac{1}{4}L$, we let $h \eqdef \eps^{1/N}$ and $b \eqdef \frac{\eps}{L}$, and we consider $P$ defined as
\[ P(x) = \prod_{j = 1}^N \Delta\pa{b, h; x_j - \frac{1}{4}} + Q(x), \]
where $Q$ is any $L$-Lipschitz function with $\spt Q \subseteq [\frac{1}{2}, 1]^N$, $\int Q = 1 - \pa{bh}^N$, and $\Delta$ is the usual tent function introduced in \eqref{eq:tent-function}. Notice that the product of tent functions is $L$-Lipschitz, because in a hypercube of size $\frac{\eps}{L}$ it has a peak of height $\eps$. The first marginal of this coupling is
\begin{ieee*}{rCl}
	\pi_1 P(x_1) &=& (bh)^{N-1} \Delta\pa{b, h; x_1 - \frac{1}{4}} + \pi_1Q(x_1) \\
	&=& \Delta\pa{b, b^{N-1}h^N; x_1 - \frac{1}{4}} + \pi_1Q(x_1).
\end{ieee*}
	
We now take any $\rho_1$ such that $\spt \rho_1 \subseteq [\frac{1}{2}, 1]$ with $\norm{\rho_1 - \pi_1P}_\infty = b^{N-1}h^N$. Crucially,
\[ b^{N-1}h^N = \frac{\eps^N}{L^{N-1}} = N^N \sigma(\eps), \]
hence $\rho_1$ satisfies the constraint (ii). However, for every $P'$ such that $\pi_1P' = \rho_1$, necessarily $\spt P' \subseteq [\frac{1}{2}, 1] \times [0,1]^{N-1}$. Hence
\[ \norm{P - P'}_\infty \geq \sup_{x \in [0, \frac{1}{2}] \times [0,1]^{N-1}} P(x) = P\pa{\frac{1}{4}, \dotsc, \frac{1}{4}} = h^N = \eps. \]

\nocite{ambrosio2013user}

\bibliographystyle{plain}
\bibliography{biblio.bib}

\end{document}